\documentclass{amsart}
\usepackage{amssymb,latexsym}
\usepackage{enumerate}
\usepackage{color}
\usepackage{graphicx}

\makeatletter
\@namedef{subjclassname@2010}{%
  \textup{2010} Mathematics Subject Classification}
\makeatother

\newtheorem{thm}{Theorem}[section]
\newtheorem{cor}[thm]{Corollary}
\newtheorem{lem}[thm]{Lemma}

\newtheorem{que}{Question}[section]

\theoremstyle{definition}
\newtheorem{defn}[thm]{Definition}
\newtheorem{rem}[thm]{Remark}

\numberwithin{equation}{section}

\DeclareMathOperator{\dist}{dist}
\DeclareMathOperator{\diam}{diam}
\DeclareMathOperator{\End}{End}
\DeclareMathOperator{\Per}{Per}
\DeclareMathOperator{\Int}{Int}
\DeclareMathOperator{\id}{id}

\newcommand{\D}{\mathcal{D}}
\newcommand{\C}{\mathcal{C}}

\newcommand{\N}{\mathbb{N}}

\newcommand{\orbp}{\mbox{Orb}^+}

\newcommand{\set}[1]{\left\{#1\right\}}
\newcommand{\eps}{\varepsilon}
\newcommand{\ra}{\rightarrow}
\renewcommand{\mod}{\mbox{mod }}

\newcommand{\fun}[3]{#1 \colon #2 \to #3}

\newcommand{\m}[1]{\mathbb{#1}}

\newcommand{\cl}[2][X]{\mathrm{Cl}_{#1}\!\left(#2\right)}
\newcommand{\ord}[2][X]{\mathrm{ord}_{#1}\!\left(#2\right)}
\newcommand{\inte}[2][X]{\mathrm{Int}_{#1}\!\left(#2\right)}
\newcommand{\bd}[2][X]{\mathrm{Bd}_{#1}\!\left(#2\right)}

\numberwithin{enumi}{section}
\begin{document}

\title{Periodic points and transitivity on dendrites}
\author[G. Acosta]{Gerardo Acosta}
\address[G. Acosta]{Inst\'{\i}tuto de Matem\'{a}ticas, Universidad Nacional Aut\'{o}noma de M\'{e}xico, Ciudad Universitaria, D.F. 04510, Mexico}
\author[R. Hern\'andez-Guti\'errez]{Rodrigo Hern\'andez-Guti\'errez}
\address[R. Hern\'andez-Guti\'errez]{UAP Cuautitl\'an Izcalli, Universidad Aut\'onoma del Estado de M\'exico. Paseos de las Islas S/N, Atlanta 2da. secci\'on. 54740, Cuautitl\'an Izcalli, M\'exico.  }
\email{rjhernandezg@uaemex.mx}
\author[I. Naghmouchi]{Issam Naghmouchi}
\address[I. Naghmouchi]{University of Carthage, Faculty of Sciences of Bizerte,
Department of Mathematics, Jarzouna, 7021, Tunisia}
\email{issam.nagh@gmail.com}
\author[P.~Oprocha]{Piotr Oprocha}
\address[P. Oprocha]{AGH University of Science and Technology\\
Faculty of Applied Mathematics\\
al. A. Mickiewicza 30, 30-059 Krak\'ow,
Poland -- and -- Institute of Mathematics\\ Polish Academy of Sciences\\ ul. \'Sniadeckich 8, 00-956 Warszawa, Poland} \email{oprocha@agh.edu.pl}

\begin{abstract}
We study relations between transitivity, mixing and periodic points on dendrites. We prove that when there is a point with dense orbit which is not an endpoint,
then periodic points are dense and there is a terminal periodic decomposition (we provide an example of a dynamical system on a dendrite with dense endpoints satisfying this assumption). We also show that it may happen that all periodic points except one
(and points with dense orbit) are contained in the
(dense) set of endpoints. It may also happen that dynamical system is transitive but there is a unique periodic point, which in fact is the unique fixed point. We also prove that on almost meshed-continua (a
class of continua containing topological graphs and dendrites with closed or countable set of endpoints), periodic points are dense if and only if they are dense for the
map induced on the hyperspace of all nonempty compact subsets.
\end{abstract}
\maketitle
\section{Introduction}

Topological transitivity is a global property of dynamical system. Under the assumption of transitivity the whole space becomes a global
attractor, making the dynamics indivisible in some sense. It also represents a minimal requirement to be able to speak about mixing in the dynamics.
Hence, it is not surprising that transitivity is an important ingredient of
global (and topological) definitions of chaos \cite{BlaSurv}.
One of the most known definitions of chaos was introduced by Devaney in his book \cite{Dev}. In this definition, a map is chaotic when it is transitive, sensitive
and contains a dense subset consisting of periodic points. It was soon noticed, that when periodic points are dense and the space is infinite (in particular, when the space is a non-degenerate continuum), transitivity implies sensitivity (see \cite{BanksDev} or \cite{GWSens}). It is well know (and not hard to prove \cite{hko_short}) that
for dynamical systems on the unit interval, transitivity implies that periodic points are dense. In other words, each transitive dynamical system on the unit interval is chaotic in the
sense of Devaney. It is also known that if the dynamical system $([0,1],f)$ is transitive but not mixing, then $[0,1]$ decomposes in two proper subintervals $I_0=[0,p]$, $I_1=[p,1]$
such that $p$ is a fixed point, $f(I_i)=I_{i+1\; (\text{mod }2)}$ and $f^2|_{I_i}$ is mixing (for $i=0,1)$. It can also be proven (see \cite{CH}), that on the unit interval, closure of the set of periodic points coincides with the closure of the set of recurrent points (so-called $\overline{R}=\overline{P}$ property).
In general, dynamics on the unit interval is well understood and many beautiful and surprising results can be proved in that context \cite{BC,ALM} (e.g. the celebrated Sharkovsky theorem). An interesting exposition on transitivity and its connection to other notions, with special emphasis on one-dimensional dynamics, can be found in the survey paper by Kolyada and Snoha \cite{KolSn}.

A natural questions that comes to mind when considering one-dimensional dynamics is to what extent these results can be generalized beyond the unit interval and what are the minimal assumptions that suffice
for these results to hold. A natural directions of research is to consider a class of ''similar'' one-dimensional continua and actually it is the way how it was done in the literature. Many authors investigated dynamical properties (related to density of periodic points) on topological trees, the circle and next on topological graphs. In many cases (after appropriate adjustment)
this direction was successful (e.g. see appendix in \cite{ALM}). A good example of necessary adjustments is the $\overline{R}=\overline{P}$ property. It is still true when we consider topological trees, but on topological graphs it is no longer valid (e.g. consider an irrational rotation of the unit circle). However it is possible to prove
in that case that $\overline{\mbox{Rec}(f)}=\mbox{Rec}(f)\cup \overline{\mbox{Per}(f)}$, where $\mbox{Per}(f)$ and $\mbox{Rec}(f)$ denote the set of periodic and recurrent points, respectively (see \cite{Blokh} or more recent \cite{MS}).
Of~course, while the result looks similar, usually there was necessary an essential modification of the arguments in the proof. In some other cases only partial answers are
known (e.g. characterization of the structure of periods of maps on topological graphs is far from being complete). In spite of technical difficulties, in the case of topological graphs dynamical properties remain to large extent ''similar'' to those known from the unit interval.

The first serious problems arise when we consider dendrites. On these (one-dimensional) continua various counterexamples to facts known from the unit interval (or topological trees) can be constructed and concrete properties hold only in special cases. For example, if a dendrite has countable set of endpoints then the $\overline{R}=\overline{P}$ property holds (see \cite{Ma2}), but when the set of endpoints points is uncountable then a counterexample can be constructed \cite{Kato}.
It was recently proved by Dirb\'{a}k, Snoha and \v{S}pitalsk\'y in \cite{DSS} that in the case of dynamical systems on continua with a free interval (in particular, on dendrites whose endpoints are not dense),
a transitive map always has dense set of periodic points and has a terminal periodic decomposition. Furthermore, an appropriate iterate restricted to any of the sets in the decomposition is mixing. On the other hand, recently there was constructed an example of a dendrite and a weakly mixing map on it which is not mixing \cite{MoHo}.
It is also known that there are dynamical systems on dendrites without a maximal $\omega$-limit set \cite{Kocan}, while for dynamical system
on a topological graph, every $\omega$-limit set is contained in a maximal one.

Motivated by the above mentioned results for topological trees (and difficulties in the case of dendrites) we state the following questions for the future research.
\begin{que}
Does every transitive dynamical system on a dendrite admit a terminal periodic decomposition?
\end{que}
\begin{que}
Provide a complete characterization of transitive dynamical
systems on dendrites with dense periodic points.
\end{que}

While we do not know complete answers to the above questions, we are able to provide some partial answers.
In Theorem~\ref{main} we prove that if a dynamical system on a dendrite has at least one point with dense orbit which is not
an endpoint then the set of periodic points of $f$ is dense in $D$.
We also prove that it may happen that there is a unique periodic point (thus a fixed point) but the map is transitive (construction from \cite{MoHo}
serves as an example), or the set of periodic points (as well as point with dense orbit), while dense, is contained in the set of endpoints.
We also provide an example of a map (and a dendrite with dense set of endpoints) which satisfies the assumptions of Theorem~\ref{main}.
This leads us to another open question.

\begin{que}
Does there exist a dendrite $D$ and a mixing dynamical system on it which does not have dense periodic points?
\end{que}

In Section~\ref{sec:decomp} we prove that a totally transitive dynamical system on a dendrite which has periodic points
that are not end points at the same time must be mixing. We also prove that at least one point with dense orbit which is not
an endpoint is sufficient for terminal decomposition.

Finally, we prove that on almost meshed continua, periodic points of $(X,f) $ are dense iff they are dense for the induced dynamical system $(2^X,2^f)$ on the space of compact subsets.
the class of almost meshed continua, among others, contains all topological graphs and dendrites with closed set of endpoints.
It is worth mentioning here, that the dendrite presented by Kato \cite{Kato} is an almost meshed continuum
such that $(X,f)$ has a unique periodic point, while $(2^X,2^f)$ has two periodic points.

\section{Preliminaries}
Let $(X,d)$ be a compact metric space. The space of all continuous maps $f\colon X\to X$ will be always endowed with the complete metric
$\rho(f,g)=\sup_{x\in X} d(f(x),g(x))$.
A \emph{continuum} is a compact, connected and nonempty metric space.
Closure of a set $A\subset X$ is denoted $\overline{A}$ or $\cl A$ and boundary of $A$ is denoted $\bd A$.

An arc is a space homeomorphic to the interval $[0,1]$.
If $A$ is an arc in $X$ with endpoints $a$ and $b,$ then $A$ is said to be a \emph{free arc} in $X$ if the set
$A\setminus\{a,b\}$ is open in $X$.

\subsection{Continua and hyperspaces}

Let $X$ be a continuum, let $p\in X$ and let $\beta$ be a cardinal number.
We say that $p$ is of order $\beta$, denoted $\ord p$, if for each open neighborhood $U$ of $p$ and each cardinal number $\alpha<\beta$
there is a neighborhood $p\in V\subset U\subset X$ such that $\alpha< \# \bd V\leq \beta$.
If $\ord p=1$ then we say that $p$ is an \emph{endpoint} and if $\ord p>2$ then we say that $p$
is a \emph{branch point}.

 A continuum
$S$ is a star if there is a point $a\in S$ such that $S$ can be
presented as the countable union $S=\cup_{n=1}^{+\infty} B_n$ of
arcs $B_n$, each having $a$ as an end point and satisfying $\lim
diam(B_n)=0$ and $B_m\cap B_n=\{a\}$ when $m\neq n$. We will say that $B_n$
is a \emph{beam} of star $S$.

A \emph{(topological) graph} is a continuum which can be written as the union of finitely many arcs, any two of
which are either disjoint or intersect only in one or both of their endpoints. A \emph{tree}
is any graph without cycles.

A continuum $D$ is a \emph{dendrite} if any two distinct points in $D$ can be separated by a third point in $D$. An accessible presentation
of dendrites and their basic properties can be found in \cite{nadler92}. If $X$ is a dendrite, then we denote by $\End(X)$ the set of \emph{endpoints} of $X$.

Let $D$ be a dendrite and $a,b\in D$. We will denote the unique arc between $a$ and $b$ by $[a,b]$. We define
$$
[a,b) = [a,b] - \{b\}, \hspace{.3cm} (a,b] = [a,b] - \{a\} \hspace{.3cm} \mbox{and} \hspace{.3cm}
(a,b) = [a,b] - \{a,b\}.
$$
If $D$ is a dendrite, we will write $x\leq y\leq z$ if $x,y,z\in D$ and $y\in [x,z]$.
This relation can be extended to any number of points contained in an arc.
Clearly if $a<p<b$, $b<q<c$ and $[a,b]\cap [b,c]=\set{b}$ then $a<p<b<q<c$.

For any subcontinuum $Y$ of $D$ we define $r\colon D \to Y$ by letting $r(x)$ to be a unique point which is a point of any arc from $x$
to any point $y\in Y$. It is known that function $r$ is well defined and continuous (and, hence, is a retraction of $D$ onto $Y$). We say that $r$ is the \emph{first point map for $Y$} and denote it for convenience by $\pi_Y=r$.
By Theorem~10.27 in \cite{nadler92}, for every dendrite $D$ and every $\eps$ there exists a tree $Y_\eps\subset D$ such that
the first point map $\pi_{Y_\eps}$ satisfies $\rho(\pi_{Y_\eps},\id)<\eps$.

For a set $Y,$ the identity on $Y$ is denoted by $1_Y.$ The interior and the closure of subset $A$ of a topological
space $X$ are denoted by $\inte{A}$ and $\cl{A},$ respectively.

If $X$ and $Y$ are continua and $\fun{f}{X}{Y}$ is a continuous function, then $f$ is said to be \emph{weakly
confluent} if for every subcontinuum $Q$ of $Y$ there exists a continuum $C$ in $X$ such that $f(C) = Q.$

For a continuum $X$ we consider the hyperspace
$$
2^X = \{A \subset X \colon A \mbox{ is non-empty and closed in } X\}
$$
\noindent endwed with the Hausdorff metric (\cite[Theorem 4.2, p. 53]{nadler92}). If $A \subset X,$ we define
$\langle A \rangle = \{B \in 2^X \colon B \subset A\}.$ It is known that if
$A$ is an open subset of $X,$ then $\langle A \rangle$ is open in $2^X$ (\cite[Theorem 4.5, p. 54]{nadler92}).

If $\fun{f}{X}{Y}$ is a continuous map, then we can consider the function $\fun{2^f}{2^X}{2^Y}$ defined,
for $A \in 2^X,$ by $2^f(A) = f(A).$ It is known that $2^f$ is continuous.

\subsection{Dynamical systems}
A \emph{dynamical system} is a pair $(X,f)$ consisting of compact metric space $(X,d)$ and a continuous map $f\colon X\ra X$.
For such system we define
$f^1 = f$ and $f^{n+1} = f^{n} \circ f,$ for each $n \in \m{N}.$
The \emph{orbit of $x\in X$} is the set $\orbp(x)=\set{f^n(x)\: :\: n\in \N}$. We say that $f$ is \emph{minimal} if the orbit of every $x\in X$ is a dense subset of $X$. A subset $M\subset X$ is a minimal set if and only if $(M,f_{|M})$ is a minimal system.

A point $x$ is \emph{periodic} if $f^n(x)=x$ for some $n>0$. When $n=1$ we say that $x$ is a \emph{fixed point}.
The \emph{period} of $p$ is the
least natural number $m$ such that $f^m(p) = p.$ If $(X,f)$ is a dynamical system, we denote by
$\mbox{Per}(f)$ the set of periodic points of $f.$

A dynamical system $(X,f)$ is \emph{transitive} (respectively \emph{mixing}) if for any pair of nonempty open sets $U, V \subset X$ there exists $n > 0$ ($N>0$) such that $f^n(U) \cap V\neq \emptyset$ (for all $n\ge N$, respectively), \emph{totally transitive} if $f^n$ is transitive for all $n\geq 1$, \emph{weakly mixing} if $f\times f$ is transitive on $X\times X$, and \emph{exact} if for any non-empty open subset $U$ of $X$ there is $N>0$ such that $f^N(U)=X$.
It is well known that if $(X,f)$ is transitive then the set $\set{x : \overline{\orbp(x)}=X}$ is \emph{residual} in $X$ (i.e the intersection of countably many sets with dense interiors).
An interesting exposition on transitivity, including dynamical systems in dimension one, can by found in survey paper by Kolyada and Snoha
\cite{KolSn}.

\section{Transitivity implies dense periodic points in some cases}

The following fact is folklore. We present its proof for completeness.

\begin{lem}\label{lem:res}
Let $D$ be a dendrite. If $\End(D)$ is dense in $D$ then $\End(D)$ is residual.
\end{lem}
\begin{proof}
First note that since $D$ contains no isolated points, for any finite subset $A\subset D$ the set $\End(D)\setminus A$ is dense in $D$.
 We can present $D$ as the closure of the  union of an increasing sequence of trees $\set{T_n:n\in \N}$.
 Each tree $T_n$ contains at most finitely many points from $\End(D)$ and so $D\setminus T_n$ contains $\End(D)\setminus A_n $ where $A_n=\End(T_n)\cap \End(D)$. But $A_n$ is finite and so $D\setminus T_n$ is an open dense subset of $D$.
 If $q\in D\setminus \End(D)$ then $D\setminus \set{q}$ consists of two nonempty open sets.
 In particular, there is $n$ such that $T_n$ intersects these two open sets, and so $q\in T_n$.
 This implies that
 $$
\bigcap_{n=1}^\infty D\setminus T_n = D\setminus \bigcup_{n=1}^\infty T_n\subset \End(D) \setminus \bigcup_{n=1}^\infty T_n\subset \End(D)
 $$
 is residual, finishing the proof.
\end{proof}

\begin{cor}\label{cor33}
Let $D$ be a dendrite and assume that $f\colon D\to D$ is transitive. If $\End(D)$ is dense in $D$ then $\set{x\in \End(D) : \overline{\orbp(x)}=D}$ is residual in $D$ (in particular, is dense in $D$).
\end{cor}

Unfortunately, it may happen that $\set{x\in D : \overline{\orbp(x)}=D}\subset \End(D)$. Example of such a dendrite and a transitive map acting on it is
presented in Section~\ref{sec:endpointstransitivity}. In what follows, we will be concerned with a special class of maps on dendrites with dense endpoints.


The following fact is stated in \cite[Lemma 2.3]{Ma2} and is an easy consequence of \cite[Theorem~10.27(5)]{nadler92}.
We present its proof for completeness.

\begin{lem}\label{l2}
Let $(C_{i})_{i\in\mathbb{N}}$ be a sequence of connected (not necessarily closed) subsets of
a dendrite $(D,d)$. If $C_{i}\cap C_{j} = \emptyset$ for all $i\neq
j$, then $$\lim_{n\to +\infty}\mathrm{diam}(C_{n})=0.$$
\end{lem}
\begin{proof}
For every $\eps>0$ there is a tree $Y$ such that the first point map $\pi_Y\colon D\to Y$ satisfies $\rho(\pi_Y,\id)<\eps/3$.
Since every point in $Y\setminus \End(Y)$ disconnects $D$, we can find finitely many points $q_1,\ldots,q_n\in Y$
such that each connected component of $D\setminus\set{q_1,\ldots, q_n}$ has diameter bounded by $\eps/2$.
This shows that there are at most $n$ sets $C_i$ with $\diam (C_i)\geq \eps$ completing the proof.
\end{proof}

We will also need another result by the same authors, which is \cite[Theorem~2.13]{Ma2}.

\begin{lem}\label{lem:returnper}
Let $D$ be a dendrite and let $(D,f)$ be a dynamical system.
Let $[x, y]$ be an arc in $D$, and $U$ be the connected component
of $D\setminus \set{x, y}$ containing the open arc $(x, y)$. If
there exist $m,n>0$ such that $\set{f^m(x), f^n(y)} \subset U$
then $U$ contains a periodic point of $f$.
\end{lem}

Now we are ready to prove the main theorem of this section.

\begin{thm}\label{main}
Let $D$ be a dendrite and let $(D,f)$ be a dynamical system.
If there is $x\in D\setminus \End(D)$ such that $\overline{\orbp(x)}=X$
then the set of periodic points of $f$ is dense in $D$.
\end{thm}
\begin{proof}
First note that since $x$ has dense orbit, it is enough to prove that every open neighborhood of
$x$ contains a periodic point.

Fix any $\eps>0$ and denote by $p$ a fixed point of $f$. Let $U$ be a connected component of $D\setminus \{x\}$ that does not contain $p$.
Fix any endpoint $e\in U\cap \End(D)$ and observe that $e\neq x$.
Since the set of branch points is at most countable in each dendrite, the arc $[e,x]$ contains at most countably many such points.
In particular, the set $D\setminus [e,x]$ has at most countably many connected components, and by Lemma~\ref{l2}
all but finite of them have diameter bounded by $\eps/2$. Hence, there is $q\in U$ such that, if we denote by $V$ the connected component of $X\setminus\{x,q\}$ that contains the open arc $(x,q)$, then $\diam (V) < \eps$.
Clearly $V\subset U$, so in particular $p\not\in V$. If we denote $V_1$ the connected component of $D\setminus \{q\}$ that contains $x$ and $V_2$ the connected component of $D\setminus \{x\}$ that contains $q$ then clearly $V=V_1\cap V_2$ and as $V_1$ and $V_2$ are open in $D$, thus $V$ is also open in $X$.

The orbit of $x$ is dense, thus there exists an integer $N>0$ such that $f^N(x)\in V$ and because $p$ is a fixed point and $p\in \omega_f(x)$, then also we have $p\in\omega_{f^N}(x)$.
%
 This implies that there is $k\geq 1$ such that $f^{kN}(x)\in U$ and $f^{(k+1)N}(x)\not\in U$.

The map $f^{kN}$ is continuous, hence
$$
x\in [f^{kN}(x),f^{(k+1)N}(x)]\subset f^{kN}([x,f^{N}(x)])
$$
and since $x$ is not periodic, there is $y\in (x,f^{N}(x))$ such that $f^{kN}(y)=x$.
Clearly $y\in V$, and there are two possibilities:
\begin{enumerate}[(i)]
  \item If $y\in [q,x]$, we denote by $W$ the connected component of $D\backslash \{y,x\}$ containing $(y,x)$. By the definition, $W$ is contained in $V$ (hence an open subset of $V$) and by the fact that $x,y$ have dense orbits, there are $m,n>0$ such that $f^n(x),f^m(y)\in W$, and Lemma~\ref{lem:returnper} implies that there
      is a periodic point in $W\subset V$.
  \item If $y\notin [q,x]$ then $y$ and $f^N(x)$ are contained in the same connected component of $D\setminus [q,x]$, in particular the connected component $W'$ of $D\backslash \{y,f^N(x)\}$ containing $(y,f^N(x))$ is contained in $V$. Applying Lemma~\ref{lem:returnper}, we obtain a periodic point in $W'\subset V$.
  \end{enumerate}
We proved then that for each $\eps>0$, there is a periodic point $z$ with $d(x,z)\leq\eps$ and hence the proof is completed.
\end{proof}

As we will see later in section 5, there is transitive dendrite map with only one periodic point, hence by Theorem \ref{main}, any point with dense orbit in this example must be an endpoint.
\begin{cor}
Let $D$ be a dendrite and $f:D\to D$ be a continuous map. If $f$ is a transitive map without dense set of periodic point then any point $x$ with dense orbit is an endpoint of $D$.
\end{cor}

\section{Transitivity and periodic decompositions on dendrites}\label{sec:decomp}
In this section we follow approach to periodic decompositions introduced by Banks in \cite{Ban97}. It is worth mentioning that decompositions
on various one-dimensional continua were considered by numerous authors. In this context, especially appropriate it is to mention work Alsed\'{a}, del Rio and Rodr\'{\i}guez \cite{ARR} on periodic decompositions on topological graphs (see also \cite{BlokhSplit} and \cite{BM}).
The situation known for topological graphs (trees) is probably the best possible that can be expected on dendrites.

A set $D \subset X$ is \emph{regular closed} it it is the closure of its
interior.
A \emph{regular periodic decomposition (of length $m$)} for a dynamical system
$(X,f)$ is a finite sequence $\D = (D_0, \ldots ,D_{m-1})$ such that:
\begin{enumerate}
\item $\D$ is a cover of $X$ consisting of regular closed sets,
\item $f(D_i) = D_{i+1(\mod m)}$ for $i=0,\ldots,  m-1$
\item\label{c:dec3} $D_i \cap D_j$ is nowhere dense in $X$ for $i \neq j$.
\end{enumerate}
Notice that for any dynamical system $(X,f)$, the space $X$ forms  a trivial regular periodic decomposition of length $1$.
Using \eqref{c:dec3} and the fact that the
boundary of a regular closed set is nowhere dense we easily get that:
\begin{enumerate}[(i)]
\item \label{RPD1} $\Int(D_i) \cap D_j = \emptyset$ for $i \neq j$,
\item \label{RPD2} the boundary of each $D_i$ is nowhere dense,
\end{enumerate}
If $\D = (D_0, \ldots ,D_{m-1})$ is a regular periodic decomposition for a transitive dynamical system $(X,f)$ then (see \cite[Lemma 2.1
and Theorem 2.1]{Ban97}) we have the following:

\begin{enumerate}[(i)]
\setcounter{enumi}{2}
\item \label{RPD3} $f^l(D_i) = D_{i+l (\mod m)}$ for $i=0,1,\ldots, m-1$ and $l\geq 0$,
\item \label{RPD4} $f^{-l}(\Int(Di))\subset Ã‚Â\Int D_{i-l(\mod m)}$ for $i=0,1,\ldots, m - 1$ and $l \geq 0$,
\item \label{RPD5} $D = \bigcup_{i\neq j}D_i \cap D_j$
is closed, invariant and nowhere dense,
\item \label{RPD5}  $f^m$ is transitive on each $D_i$.
\end{enumerate}
If all the sets of $\D$ are connected we say that $\D$ is \emph{connected}. By \eqref{RPD3} and the fact that the
continuous image of a connected set is connected we get that $\D$ is connected if and only if
at lest one of the sets $D_i$ is connected.
By \cite[Theorem 2.4]{Ban97}, if $(X,f)$ is transitive but $(X,f^n)$ is not transitive for some $n>0$ then $(X,f)$ has a regular periodic decomposition of length $p>0$
where $p$ is a prime number dividing $n$.

Assume now that $\C=(C_0,\ldots, C_{n-1})$ and $\D = (D_0, \ldots ,D_{m-1})$ are regular periodic decompositions for $(X,f)$. We say that
$\C$ \emph{refines} (or is a \emph{refinement} of) $\D$ if every $C_i$ is contained in some $D_j$. Note that in such a case each element of $\D$
contains the same number of elements of $\C$, so $n$ is a multiple of $m$.
It was proved in \cite[Theorem 6.1]{Ban97} that if $X$ locally connected then every regular periodic decomposition for $(X,f)$ has a connected
refinement. By \cite[Lemma 3.2]{Ban97} if $(X,f)$ has regular periodic decompositions $\C$ of length $m$  and $D$ of length
$n$, then $(X,f)$ has a regular periodic decomposition of length $k = \text{lcm}(m, n)$ which is a
refinement of both $\C$ and $\D$. We will call a regular periodic decomposition \emph{terminal} if it is of maximal
length. There are situations where terminal decomposition exist. In such a case we have the following characterization, which is \cite[Theorem~3.1]{Ban97}.

\begin{thm}
A regular periodic decomposition $\D = (D_0, \ldots, D_{n-1})$ for a
transitive dynamical system $(X,f)$ is terminal iff $f^n$ is totally transitive on each $D_i$.
\end{thm}

Note that every connected subset of dendrite is arcwise connected and every subcontinuum of dendrite is dendrite (for a
useful summary of properties of dendrites and their equivalent characterizations the reader is referred to \cite{Chara98}).


\begin{thm}\label{thm:decomp_partial}
If $D$ is a dendrite and $(D,f)$ is transitive and there is no terminal decomposition for $f$
then for any $q\in D\setminus \End(D)$ there is a connected regular periodic decomposition $\D=(D_0, \ldots, D_{m-1})$ such that
$q\in D_i \cap D_j$ for some $i\neq j$.
\end{thm}
\begin{proof}
Since there is no terminal decomposition, there are numbers $n_1,n_2,\ldots$, where each $n_j\geq 2$ and connected regular periodic decompositions

$$
\D_k=\set{D_{i_1,i_2,\ldots, i_k} : i_j\leq n_j}
$$
of lengths $n_1n_2\ldots n_k$ such that $\D_{k+1}$ is a refinement of $\D_k$ and $D_{i_1,i_2,\ldots,i_k,i_{k+1}}\subset D_{i_1,i_2,\ldots,i_k}$. Fix $q\in D\setminus \End(D)$. We will prove that there is $k$ such that $q\in D_{i_1,i_2,\ldots,i_k}\cap D_{j_1,j_2,\ldots,j_k}$ for some $(i_1,i_2,\ldots,i_k)\neq(j_1,j_2,\ldots,j_k)$. Assume that, on the contrary, there are indices $r_1,r_2,\ldots$ such that $q\in\Int{D_{r_1,r_2,\ldots,r_k}}$ for every $k$. We shall arrive to a contradiction.

Every dendrite has the fixed point property, so there exists a fixed point $p\in D$. For every $k$ and $j_i\in\{0,\dots,n_i-1\}$ where $1\leq i\leq k$, we cannot have $p\in \Int D_{j_1,j_2,\ldots, j_k}$ because then $p\in \Int D_{j_1,j_2,\ldots, j_k} \cap f(D_{j_1+1 (\mod n_1)},j_2,\ldots, j_k)=\emptyset$. Then the only possibility is that $p\in \bigcap D_{j_1,j_2,\ldots, j_k}$. Clearly $p\neq q$. Let $V$ be a connected component (hence an open set) of $X\setminus \set{q}$ which does not contain $p$.

Fix $k$ for the moment, we will now prove that $V\subset D_{r_1,r_2,\ldots,r_k}$. Assume that this is not the case so there is $z\in V\cap D_{i_1,i_2,\ldots,i_k}$ where $(i_1,i_2,\ldots,i_k)\neq(r_1,r_2,\ldots,r_k)$. Since $p,z\in D_{i_1,i_2,\ldots,i_k}$ and $D_{i_1,i_2,\ldots,i_k}$ is arcwise connected, $[p,z]\subset D_{i_1,i_2,\ldots,i_k}$. Since $q\in[p,z]$, $q\in D_{i_1,i_2,\ldots,i_k}$ which is impossible because we are assuming that $q\in\Int{D_{r_1,r_2,\ldots,r_k}}$ and this set is disjoint from $D_{i_1,i_2,\ldots,i_k}$.

As $V$ is open in $X$ then $V\subset \Int D_{r_1,r_2,\ldots,r_k}$ for any $k$. By transitivity, there is $m>0$ such that $f^m(V)\cap V\neq \emptyset$. Choose $k\in\mathbb{N}$ such that $n_1n_2...n_k>m$. The number $m$ is smaller than the length of the decomposition $\D_k$, hence
$$
f^m(V)\cap V \subset f^m(D_{r_1,r_2,\ldots,r_k})\cap \Int D_{r_1,r_2,\ldots,r_k}=\emptyset
$$
which is a contradiction.

\end{proof}

\begin{rem}
Assume that $\C=(C_0,\ldots, C_{n-1})$ and $\D = (D_0, \ldots ,D_{m-1})$ are regular periodic decompositions for a dynamical system $(D,f)$
on a dendrite $D$ and assume that
$\C$ refines $\D$. If $i,j,k$ are such that
$D_i\cap D_j\neq \emptyset$ and $C_k\subset D_i$ then $C_k\cap D_j\neq \emptyset$ (see \cite{Ban97}). So if there is a dendrite without terminal periodic decomposition then in the view of Theorem~\ref{thm:decomp_partial} it seems to have very special structure.
\end{rem}

\begin{cor}\label{cor:decomp}
If $(D,f)$ is a transitive dynamical system on a dendrite $D$ and there is a point $q\in D\setminus \End(D)$
with dense orbit then
there is a terminal decomposition for $(D,f)$.
\end{cor}
\begin{proof}
For any decomposition $\set{D_i}$ the set
$D=\bigcup_{i\neq j} D_i \cap D_j$
is invariant and nowhere dense, in particular  $q\not\in D$. By Theorem~\ref{thm:decomp_partial} there is a terminal decomposition.
\end{proof}

We also present an interesting property of totally transitive systems which was first proved by Banks in \cite{Ban97}.
\begin{lem}\label{totweak}
If $(X,f)$ is a totally
transitive dynamical system with dense set of periodic points, then $(X,f)$ is weakly mixing.
\end{lem}

Now, we have all the tools to prove the following fact.

\begin{thm}\label{tt:mix}
Let $f\colon D \to D$ be a totally transitive dendrite map, and suppose that $Per(f)\setminus \End(D)$
is dense. Then $f$ is mixing.
\end{thm}
\begin{proof}
Take any two nonempty open sets $U,V$. Taking a smaller set if necessary, we may assume that $U$ is connected.
There is a periodic point $p\in V\setminus \End(D)$.
Denote by $n$ the period of $p$. Since $p$ is not an endpoint, there are two open sets $V_1, V_2$
contained in two different connected components of $D\setminus \set{p}$.
By Lemma \ref{totweak}, $f$ is weakly mixing, then the $(2n+2)$-times Cartesian product of copies of $f$ is also transitive (see \cite{FU}). Hence there is $N>0$ such that
$$
(\Pi_{1\leq i\leq 2n+2}f^N)(\Pi_{1\leq i\leq 2n+2}U)\cap (\Pi_{0\leq i\leq n}f^{-i}(V_1)\times \Pi_{0\leq i\leq n}f^{-i}(V_2))\neq\emptyset.
$$
So
$f^{N+i}(U)\cap V_1\neq \emptyset$ and $f^{N+i}(U)\cap V_2\neq \emptyset$ for $i=0,1,\ldots,n$.
But since $U$ is connected we obtain that $p\in f^{N+i}(U)$ for $i=0,\ldots,n$.
Now, if we fix any $k\geq N$ then $k=N+i+jn$ for some $j\geq 0$ and $0\leq i<n$.
But then
$$
p=f^{jn}(p)\in f^{jn}(f^{N+i}(U))=f^k(U)
$$
which gives $f^k(U)\cap V\neq \emptyset$ completing the proof.
\end{proof}

\section{Endpoints and transitivity}\label{sec:endpointstransitivity}

In this section we will present a construction of a dendrite and transitive map on it, where periodic points are dense and all
are contained in endpoints of the dendrite except only one.

Let $S$ be a star with center $b_{\emptyset}$ and beams $B_n$, for $n\in\mathbb{N}$.
In the interior of each beam $B_n$, we locate a dense sequence of
points $(b_{(n,i)})_{i>0}$.\\
In every point $b_{(n,i)}$, we attach a star $S_{(n,i)}$ with center
$b_{(n,i)}$ such that $S_{(n,i)}\cap S=b_{(n,i)}$ and $diam(S_{(n,i)})<2^{-(n+i)}$. Let us call the beams of the star $S_{(n,i)}$  by  $B^{(n,i)}_m$ for $m\in\mathbb{N}$.  \\
In the interior of every beam $B^{(n,i)}_m$, we locate similarly as
above a dense sequence of points $(b_{(n,i)\times(m,j)})_{j>0}$.
\medskip

We follow this process by induction; For any $k\in\mathbb{N}$, and
for any $\alpha\in (\mathbb{N}^2)^{k}$ we attach a star $S_{\alpha}$ in
the point $b_{\alpha}$ such that $diam(S_{\alpha})<2^{-(\sum_{i=1}^kn_i+j_i)}$ where $\alpha=(n_1,j_1)\dots (n_k,j_k)$. We note the beams of $S_{\alpha}$ by  $B^{\alpha}_{m}$,
$m\in\mathbb{N}$. In the interior of every beam $B^{\alpha}_{m}$, we
locate similarly a dense sequence of points $(b_{\alpha\times
(m,j)})_{j>0}$. In each point $b_{\alpha\times (m,j)}$, we attach a
star $S_{\alpha\times (m,j)}$ such that $S_{\alpha\times (m,j)}\cap
S_{\alpha}=b_{\alpha\times (m,j)}$.\\

Then, we let $X:=\overline{S\cup \bigcup_{k>0}( \cup_{\alpha\in
(\mathbb{N}^2)^{k}}S_{\alpha})}$. The set of branch points of $D$ is
$\{b_{\emptyset}\}\cup \cup_{k\in\mathbb{N}}\{b_{\alpha}$, $\alpha\in
(\mathbb{N}^2)^{k}\}$. It is clear from the construction of $X$ that the set of branch point is dense and so is the set of endpoints.

For $k\in\mathbb{N}\cup\{+\infty\}$ and $i\in\{1,2,\dots,k\}$, if $\alpha\in (\mathbb{N}^2)^k$
we denote by $\alpha_i$ the pair in $\mathbb{N}^2$ that occur in the position $i$ in the sequence
 $\alpha$. Hence, we can write $\alpha=\alpha_1\alpha_2\dots\alpha_k$ if $k\neq +\infty$. For $n,k\in\mathbb{N}$ and $\alpha\in (\mathbb{N}^2)^k$, the concatenation of $\alpha$ for $n$ times, denoted by $\alpha^n$, is the sequence
 $\beta\in(\mathbb{N}^2)^{nk}$ given in the following way:  $\beta_{i}=\alpha_{i}$ for each $i\in\{1,\dots,k\}$ and $\beta_{j}=\beta_{j-k}$ for each $j\in\{k+1,\dots,nk\}$.
For $k\in\mathbb{N}$ and $k^{'}\in\mathbb{N}\cup\{+\infty\}$ such that $k<k^{'}$,
 if $\alpha\in (\mathbb{N}^2)^k$ and $\alpha^{'}\in (\mathbb{N}^2)^{k^{'}}$, we write
 $\alpha\subset \alpha^{'}$ if $\alpha_i=\alpha^{'}_i$ for $1\leq i\leq k$.
Let us denote by $\Lambda=\{\emptyset\}\cup\bigcup_{k\in\mathbb{N}}(\mathbb{N}^2)^k$.

We define the function $\sigma:\Lambda\to\Lambda$ which will be used later to describe the dynamics of the map $F$ on the set of branch points: For an $\alpha\in\Lambda$, let us denote by $\alpha_i=(n_i,j_i)$ for any $1\leq i\leq k$. Then we define $\sigma(\alpha)=\beta$ in the following way:
\begin{enumerate}[(i)]
\item If $n_1=1$ and $k>1$ then $\beta_1=(j_1+n_2-1,j_2)$ and $\beta_i=\alpha_{i+1}$ for all $1<i\leq k$.
\item If $n_1=1$ and $k=1$ then $\beta=\emptyset$.
\item If $n_1>1$ then $\beta_1=(n_1-1,j_1)$ and $\beta_i=\alpha_{i}$ for all $1<i\leq k$.
\end{enumerate}

For $k\in\mathbb{N}$ and $\alpha\in (\mathbb{N}^2)^{k}$, we denote by $X_{\alpha}$
the union of the closure of connected components of $X\setminus\{b_{\alpha}\}$ that
contains all the beams $B^{\alpha}_{m}$, for $m\in\mathbb{N}$.
Then for any non-empty open subset $U$ of $X$, there is $\alpha\in(\mathbb{N}^2)^{k}$
for  some $k\in\mathbb{N}$ such that $X_{\alpha}\subset U$.

For any $n\in\mathbb{N}$, let us denote by $D_n$, the closure of the
connected components of $D\setminus \{b_{\emptyset}\}$ that contains the beam $B_n$.

Now we are ready to define a map $F\colon X\to X$.
 We let $F_{\mid D_{n+1}}:D_{n+1}\rightarrow D_n$ be a homeomorphism which
mapping linearly each arc
$$[b_{\emptyset},b_{(n+1,j_1)(n_2,j_2)\dots(n_k,j_k)}]$$ to the arc
$$[b_{\emptyset},b_{(n,j_1)(n_2,j_2)\dots(n_k,j_k)}].$$

\begin{figure}[!htbp]
\begin{center}
\includegraphics[width=0.8\textwidth]{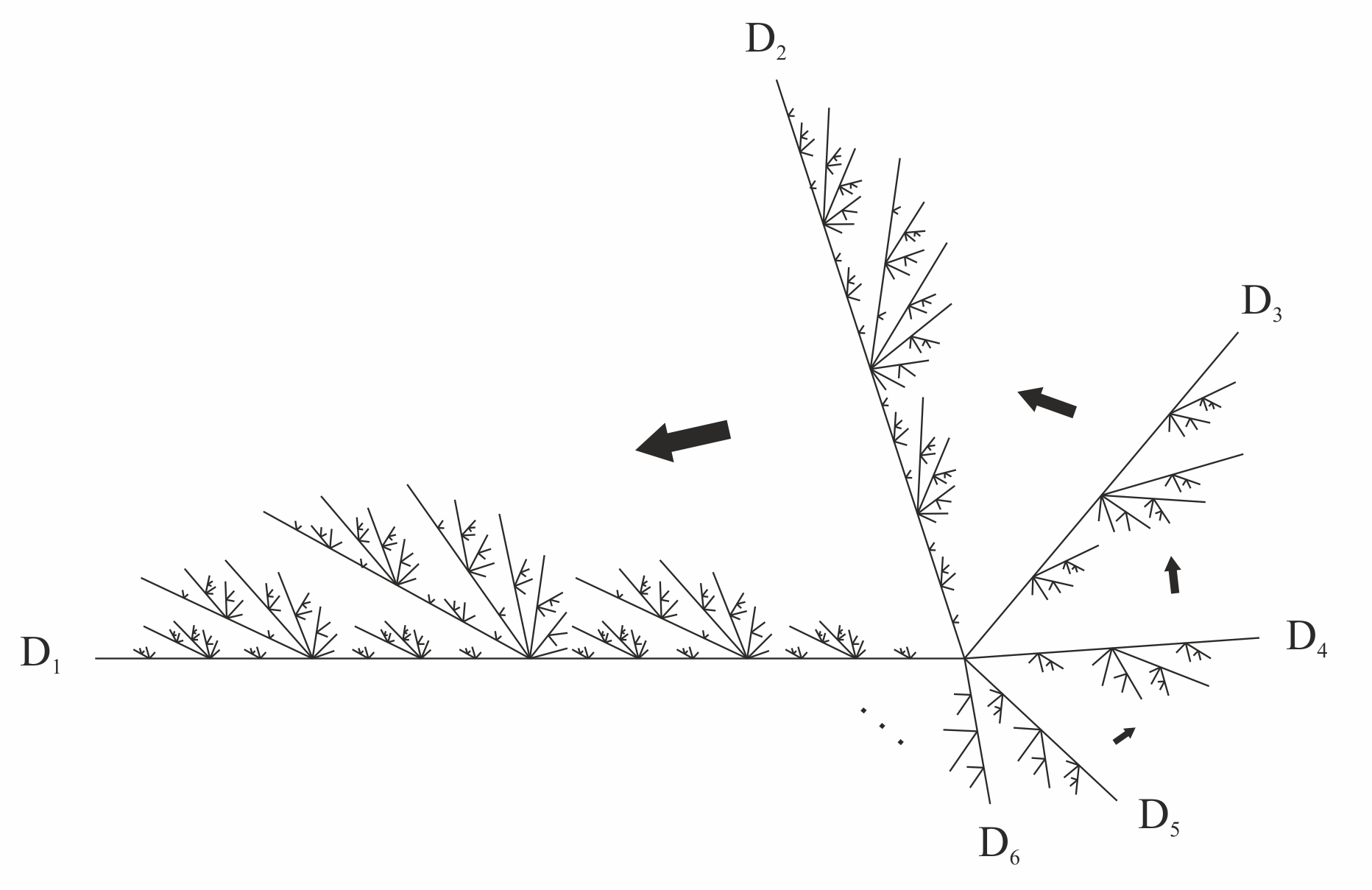} \caption{Dendrite $D$ with a sketch of action of $F_{\mid D_{n+1}}:D_{n+1}\rightarrow D_n$}
\end{center}
\end{figure}

\begin{figure}[!htbp]
\begin{center}
\includegraphics[width=0.8\textwidth]{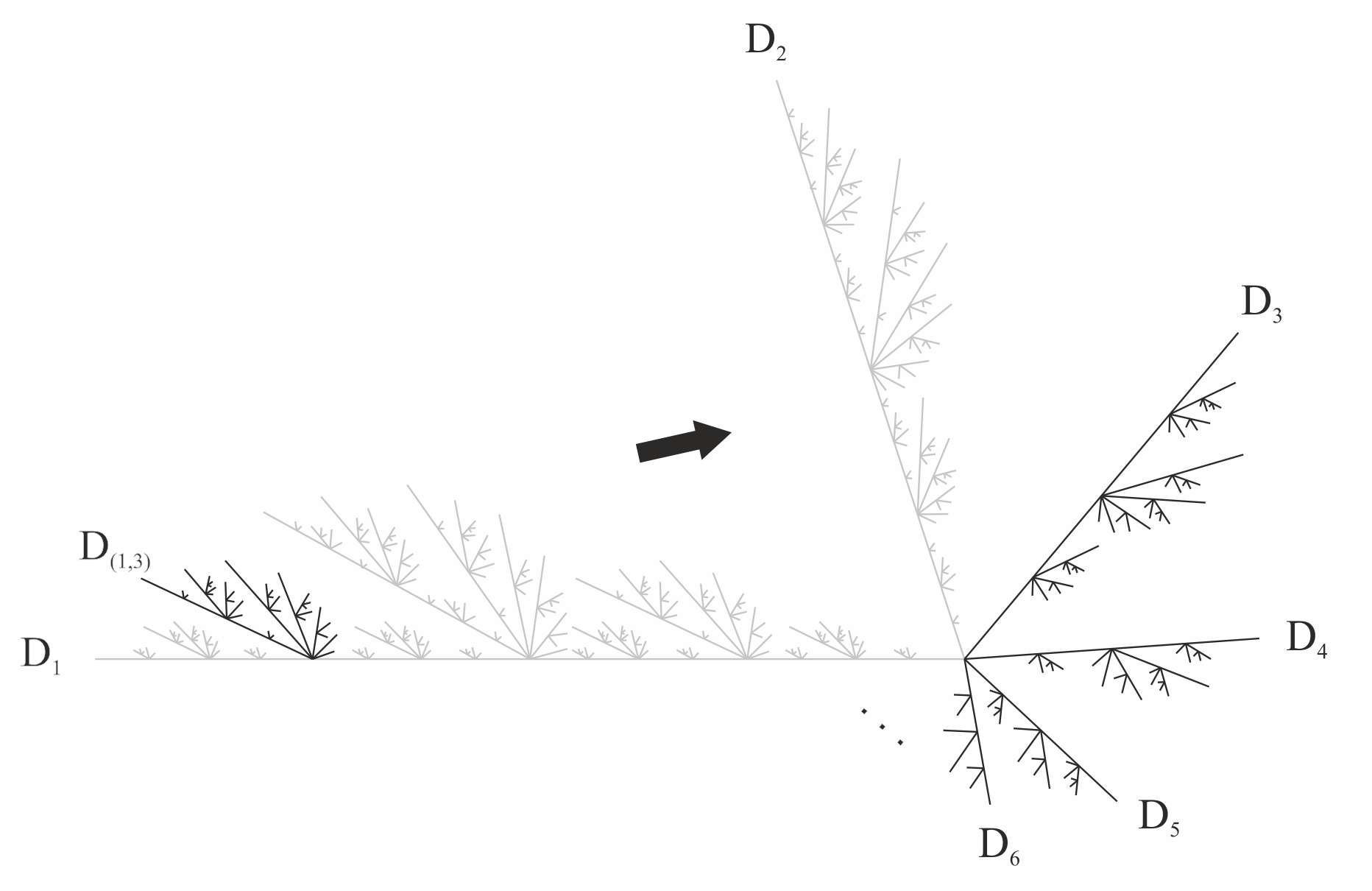} \caption{Dendrite $D$ with a sketch of action of $F_{\mid D_{(1,3)}}:D_{(1,3)}\rightarrow \bigcup_{i\geq 3}D_i$}
\end{center}
\end{figure}

%

Now we define the map $F$ on the dendrite
$D_1$:
First, we let $F(b_{\emptyset})=b_{\emptyset}$.
\medskip
Then we let $F_{\mid D_{(1,n)}}:
D_{(1,n)}\rightarrow \cup_{i\geq n} D_i$ be a homeomorphism which
mapping linearly any arc
$$[b_{(1,n)},b_{(1,n)(n_1,j_1)...(n_k,j_k)}]$$ to
the arc
$$[b_{\emptyset},b_{(n+n_1-1,j_1)...(n_k,j_k)}],$$
%
%
%

where  $n,k\in\mathbb{N}$ and
$(n_1,j_1)(n_2,j_2)...(n_k,j_k)\in(\mathbb{N}^2)^{k}$ (see Figure. 2).

 In this way, the beam $B_1$ is collapsed to the fixed point $b_{\emptyset}$ and
$F(D_{(1,n)})=\cup_{i\geq n} D_i$.

Hence by using the map $\sigma$, we have for any  $\alpha\in\Lambda $,
\begin{equation}
\label{eq:i*}
F(b_{\alpha})=b_{\sigma(\alpha)}.
\end{equation}

Notice that if $1<k\in\mathbb{N}$,
$(n_1,j_1)(n_2,j_2)...(n_k,j_k)\in\mathbb{N}^{2k}$ and
$m=(\sum_{i=1}^{k} n_i+j_i)-k$, then
\begin{equation}\label{eq1}
F^{m}([b_{\emptyset},b_{(n_1,j_1)(n_2,j_2)...(n_k,j_k)}])=b_{\emptyset}.
\end{equation}

\begin{thm}
Let $F\colon X \ra X$ be the continuous map constructed in this section.
The following conditions are satisfied.
\begin{enumerate}[(i)]
\item \label{c:exact} the map $F$ is exact,
\item \label{c:periodic}The set of periodic points of $F$ is dense in $X$ and it is included in the set of endpoints of $X$ except the point $b_{\emptyset}$.
\end{enumerate}
\end{thm}
\begin{proof}
In order to prove \eqref{c:exact} let us fix
any nonempty open set $U\subset X$. Then there exist $k\in\mathbb{N}$ and $\alpha\in(\mathbb{N}^2)^{k} $ such that $X_{\alpha}\subset U$. Let $m=(\sum_{i=1}^{k}n_i+j_i)-k$ where $(n_i,j_i)=\alpha_i$ for $i=1,\ldots,k$. By the construction of $F$ (e.g. see \eqref{eq:i*}) we easily obtain that $F^m(X_{\alpha})=X$.
This completes the proof of \eqref{c:exact}.

To prove the density of the set of periodic points it is sufficient to prove the existence of a periodic point in each $X_{\alpha}$ for any $\alpha\in(\mathbb{N}^2)^{k}$ and for any $k\in\mathbb{N}$. For any $n\in\mathbb{N}$, we have $X_{\alpha^{n+1}}\subset X_{\alpha^{n}}$.

Note that $\lim_{n\to +\infty} \diam(X_{\alpha^{n}})=0$ and hence $\bigcap_{n\geq 1}X_{\alpha^{n}}$ is a single point, say $z\in \bigcap_{n\geq 1}X_{\alpha^{n}}$. We claim that $z$ is an endpoint and $z=\lim_{n\to +\infty}b_{\alpha^{n}}$. Indeed, we have for any $n\geq 1$, $[b_{\alpha^{n}},b_{\alpha^{n+1}})\cap X_{\alpha^{n+1}}=\emptyset$ and we have also $b_{\alpha^{n+2}}\in X_{\alpha^{n+1}}$, hence $[b_{\alpha^{n}},b_{\alpha^{n+1}}]\cup [b_{\alpha^{n+1}},b_{\alpha^{n+2}}]$ is an arc. So by induction, for any $n\geq 1$, $[b_{\alpha},b_{\alpha^{n}}]$ is an arc, which implies that the sequence $(b_{\alpha^{n}})_{n\in\mathbb{N}}$ converges and hence its limit must be $z$.
If $z\not\in \End(X)$ then there is an endpoint $e\in X$ such that $z\in [b_{\alpha},e]$ and so for each $n\in\mathbb{N}$, $e\in X_{\alpha^{n}}$ since $z$ and $e$ are in the same connected component of $X\setminus\{b_{\alpha^n}\}$. This shows that $e,z\in \bigcap_{n\geq 1}X_{\alpha^{n}}$ which is a contradiction, since $z\neq e$. We conclude that $z$ is an endpoint of $X$.

Now, we are going to prove that $z$ is a fixed point of $F^m$ where $m=(\sum_{i=1}^k n_i+j_i)-k$. Fix any $\eps>0$ and let $0<\delta<\eps/2$ be such that
$d(F^m(p),F^m(q))<\eps/2$ provided that $d(p,q)<\delta$.
There is $N>0$ such that for each $n>N$, we have $d(z,b_{\alpha^{n}})<\delta$ and so for every $n>N$ we have
$$
d(F^m(z),z)\leq d(z, b_{\alpha^n})+d(F^m(z),F^{m}(b_{\alpha^{n+1}}))<\eps.
$$
Since $\eps$ was arbitrary, $z$ is a fixed point of $F^m$.
By \eqref{eq1} all cut points except $b_{\emptyset}$ are not periodic (but eventually fixed, since they are eventually send by $F$
to $b_{\emptyset}$).
\end{proof}

In next theorem we will analyze the construction presented by Hoehn and Mouron in \cite[Example~17]{MoHo}.
Recall, that a pair $(x,y)$ is \emph{proximal} in a dynamical system $(X,f)$ if $\liminf_{n\to\infty} d(f^n(x),f^n(y))=0$
and $(X,f)$ is \emph{proximal} if any pair of points in $X$ is a proximal pair.
It can be proved (see \cite{AK}) that a dynamical system $(X,f)$ is proximal if it has a fixed point which is the unique minimal subset of $X$. We will show that \cite[Example~17]{MoHo} gives an example of transitive dendrite map with a unique periodic point hence without dense set of periodic points.

\begin{thm}\label{mouronexample}
There is a dendrite $D$ and a map $F\colon D \to D$ such that:
\begin{enumerate}[(i)]
\item \label{c:exact} the map $F$ is weakly mixing but not mixing,
\item \label{c:periodic} the map $F$ is proximal (in particular it has the unique periodic point which is the fixed point).
\end{enumerate}
\end{thm}
\begin{proof}
Let $D$ be the Wa\.zewski universal dendrite (see \cite{nadler92}).
Take any infinite set $Z\subset \N$ such that for every $n>0$ there is $j>0$ such that $\set{j,j+1,\ldots,j+n}\subset \N\setminus Z$.
It was proved in \cite[Example~17]{MoHo} that there exists sub-dendrites $\set{E_j}_{j=0}^\infty$, a point $o\in D$ and a map $F\colon D\to D$
such that:
\begin{enumerate}
\item $D=\bigcup_{j=0}^\infty E_j$ and $E_i\cap E_j=\set{o}$ for each $i\neq j$,
\item $F$ is weakly mixing but not mixing,
\item $F(o)=o$ and $f(E_j)=E_{j-1}$ for $j>0$,
\item $f^n(E_0)\subset \bigcup_{n+j-1\in Z}E_j$ for every $n>0$.
\end{enumerate}
Let $M$ be a minimal subset of $D$ and fix any $x\in M$ and suppose that $x\neq o$.
Then there is $s\geq 0$ and an open neighborhood $U$ of $x$ such that $U\subset E_s$.
Since $x$ belongs to a minimal set, there exists $K>0$ such that for every $n>0$ there is
$0\leq j < K$ such that $f^{n+j}(x)\in U$.

But there exists $k\in \N$ such that $\set{k,k+1,\ldots, k+s+K+1}\cap Z =\emptyset$.
Take any integer $0\leq i \leq K$ and observe that
$$
f^{i+k+s}(x)\in f^{k+i}(E_0)\subset \bigcup_{k+i+j-1\in Z}E_j \subset \bigcup_{j>s}E_j
$$
hence $f^{i+k+s}(x)\not\in U$ for $i=0,1,\ldots,K$, which is a contradiction.
This proves that $M=\set{o}$ is the unique minimal set for the dynamical system $(D,F)$.
\end{proof}

\section{Transitivity, periodic points and branch points}\label{sec:mix:constr}

The aim of this section is to show that there is a dendrite $D$ with dense endpoints and a mixing map $f$ on it such that
assumptions of Theorem~\ref{main} are satisfied (or even more, every branch point is a periodic point). Before we start with the construction we need to introduce some
notation together with a few standard facts and techniques.

We recall that a map $f\colon X\to Y$ is \emph{monotone} if $X$ and $Y$ are
metric spaces, $f$ is continuous, and for each point $y\in Y$ its
preimage $f^{-1}(y)$ is connected. If
$X$ is a tree and there is a finite set $P\subset X$ such that for each
connected component $C$ of $X\setminus P$ the map $f|_{\overline{C}}\colon \overline{C}\to Y$
is monotone, then we say that $f$ is $P$-\emph{monotone}.

\begin{defn}
Given a dynamical system $(G,f)$ acting on a topological graph $G$ and free arcs $I,J\subset G$, we
say that $I$ \emph{covers} $J$ through $f$ (or $I$~$f$-covers~$J$,
for short) if there exists an arc $K \subset I$ such
that $f (K) = J$.
\end{defn}

Let $T$ be a tree and consider a dynamical system $(T,f)$. We say that $f$ is \emph{linear} on a
set $S\subset T$ if there is a constant $\alpha$ such that
$d(f(x),f(y))=\alpha d(x,y)$ for all $x,y\in S$ (here, as always $d$
denotes the taxicab metric on $T$). If $P\subset T$ is a finite set such that for each
connected component $C$ of $X\setminus P$ the map $f|_{\overline{C}}\colon \overline{C}\to T$
is linear, then we say that $f$ is $P$-\emph{linear}.

If $f$ is
$P$-monotone map then we call $f$  a \emph{Markov map} if $P\subset T$ is a finite set containing all
endpoints and branch points of $T$ and $f(P)\subset P$. In the above situation we call $f$ a
$P$-\emph{Markov} map for short and by \emph{$P$-basic interval} we mean any arc $J\subset T$ whose endpoints are in $P$ and there is no other point of $P$ in $J$.
If $f$ is a $P$-Markov map, then the
\emph{Markov graph} of $f$ with respect to $P$ ($P$-Markov graph of $f$ for
short) is defined as a directed graph with the set of $P$-basic
intervals as a set of vertices and with the set of edges defined by the
$f$-covering relation, that is, there is an edge from a $P$-basic interval
$I$ to $P$-basic interval $J$ in the Markov graph ($I\to J$) if and only if $J$ is
$f$-covered by $I$.

There are also few results we would like to recall for later reference.
The following fact is \cite[Corollorary 1.11]{Baldwin}.

\begin{lem}\label{lem:Markov-transitive}
Suppose $T$ is a tree and $f\colon T\to T$ is $P$-Markov and $P$-linear with
respect to some $P$ containing all endpoints and branch points of $T$. Then $f$ is transitive if and only
if the $P$-Markov graph of $f$ is strongly connected (i.e. there is a path between any two vertices of the graph) and is not a graph of a cyclic
permutation.
\end{lem}

The following fact is a small modification of result in \cite{hko}.
The main difference is that instead of estimating entropy, we want to control range of perturbation.

\begin{lem}\label{lem:small perturb}
If $f$ is a transitive $P$-Markov and $P$-linear map of a tree $T$, which is not
totally transitive, then for every $\eps>0$, there is a finite set $P'\subset T$ and a totally transitive (hence, exact)
$P'$-Markov and $P'$-linear map $f'\colon T\to T$ such that $P\subset P'$, $f|_P=f'|_P$ and
$\rho(f,f')<\eps$.
\end{lem}
\begin{proof}
First note that for any periodic point $r$ the map $f$ is also $(P\cup \orbp(r,f))$-Markov and $(P\cup \orbp(r,f))$-linear.
It is also well known that transitive maps on trees have dense periodic points (e.g. see \cite{hko_short}) and
so adding finitely many points to $P$ if necessary, we may assume that
for each $P$-basic interval $J$, the diameter of both $J$ and $f(J)$ is less than $\eps/3$.

Since $(T,f)$ is not totally transitive, then by Theorem 4.1 there exists a periodic decomposition $T_0, T_1, \ldots, T_{n-1}$ with at least two elements.
There is a fixed point $p$ in $T$, hence $p\in T_i$ for every $i$. We must have
$T_i \cap T_j=\set{p}$ for $i\neq j$, because elements of the decomposition have pairwise disjoint interiors.
But then the number of elements in each decomposition cannot exceed the number of connected components of $T\setminus \set{p}$.
Hence, without loss of generality, we may assume that $\set{T_i}_{i=0}^{n-1}$ is a terminal decomposition, which in particular implies by Theorem 4.1
that $f^n|_{T_i}$ is totally transitive for each $i=0,1,\ldots,n-1$. But since periodic points are dense
$f^n|_{T_i}$ is mixing (e.g. see \cite{hko_short}). Clearly for each $m$, $P$-linear map can have finitely many periodic points with period less than $m$ so $f^n|_{T_i}$ is exact (see \cite[Theorem~7.3]{hko}).
Obviously, each map $f^n|_{T_i}$ is also $P_i$-Markov and $P_i$-linear, where $P_i=f^{-n}(P) \cap T_i$.
There
is a point $q\neq p$ such that $f(q)=p$ and without loss of generality we may
assume that $q\in T_0$. We may also include $q$ in $P$ (when $q\not\in P$) and after this modification $f$ remains
$P$-linear and $P$-Markov.

Let $[q,t]$ be a $P$-basic interval in $T_0$. For $j=0,1,\ldots,n-1$ let $I_j=[p,z_j]$ denote a $P$ basic
interval in $T_j$ containing $p$.
Fix any points $r,s\in (q,t)$ and assume that $q<r<s<t$ with respect to ordering in $[q,t]$.
First, we put $f'(x)=f(x)$ everywhere except $(q,t)$. Next we define $f'(r)=f(t)$, $f(s)=z_0$ and make $f'$ linear on
intervals $[q,r]$, $[r,s]$ and $[s,t]$. Observe that in the Markov graph of $f'$, the set of vertices which can be reached by an edge from any of the vertices $[q,r]$, $[r,s]$ or $[s,t]$
contains all vertices that can be reached from $[q,t]$ in Markov graph for $f$. Similarly, if there is an edge from $J$ to $[q,t]$ in Markov graph of $f$, then there is also an
edge from $J$ to any of the vertices  $[q,r]$, $[r,s]$ or $[s,t]$ in the Markov graph of $f'$. Therefore the Markov graph of $f'$ is strongly connected, since the Markov graph of
$f$ is strongly connected by Lemma~\ref{lem:Markov-transitive}.
Finally, observe that
$$
f'(T_0)\supset f'([q,t])\supset f([q,t])\cup [p,z_0]\supset [p,z_0]\cup [p,z_1],
$$
and hence $f'(T_0)$ intersects interiors of both $T_0$ and $T_1$.

We claim that $f'$ is totally transitive. Repeating arguments from the beginning of the proof, there exists a terminal decomposition $D_0,D_1,\ldots, D_{k-1}$
for $f'$. If $k=1$ then $f'$ is totally transitive and there is nothing to prove. In the other case $D_i\cap D_j=\set{p}$ for each $i\neq j$
and $p$ is the unique fixed point. Then $[q,t]$ is a subset of one of the sets $D_i$, say $[q,t]\subset D_0$.
Note that for every $P$-basic interval $I$ we have $f(I)\subset f'(I)$ and hence, since each $f^n|_{T_i}$ is exact, there is $l>0$ (which is a multiple of $n$)
such that $f^l([p,z_i])=T_i$ for each $i=0,1,\ldots,n-1$, hence we have $(f')^{l+1}([q,t])\supset T_{0} \cup T_{1}$.
Similarly $(f')^{l+2}([q,t])\supset [p,z_0]\cup[p,z_1]\cup T_2$ and so $(f')^{2l+2}([q,t])\supset T_0\cup T_1\cup T_2$. Repeating this procedure inductively, we eventually obtain
$$
(f')^{(n-1)(l+1)}([q,t])\supset T.
$$
In particular $(f')^{(n-1)(l+1)}(D_0)\supset D_0 \cup D_1$ which is a contradiction.
Indeed $f'$ is totally transitive, hence exact.

The proof is finished by the fact that
$$
d(f(x),f'(x))\leq \diam([z_0,p])+\diam f([q,t])\leq \eps/3+\eps/3<\eps.
$$
\end{proof}

The next lemma provides a tool for extending a transitive map on a tree $T'$ to a tree $T$ obtained by attaching edges at its end points,
provided that the dynamics restricted to these new edges is transitive.

\begin{lem}\label{lem:tree_attach}
Let $T$ be a tree, let $T'\subset T$ be a tree invariant for a map $f\colon T\to T$.
Let $C_0, \ldots, C_{n-1}\subset T$  be closed trees with pairwise disjoint interiors and such that $f(C_i)=C_{i+1 (\text{mod }n)}$
and $f(p_i)=p_{i+1 (\text{mod }n)}$ where $\set{p_i}=C_{i}\cap T'$ is their common endpoint (not necessarily $p_i\neq p_j$).
Finally assume that there is a partition $P$ such that $f$ is $P$-Markov and $P$-linear, and $f$ is transitive
restricted to any of sets $T'$ and $\bigcup_{i=0}^{n-1} C_i$.

Then for every $\eps>0$, there is a totally transitive (hence, exact)
$P'$-Markov and $P'$-linear map $f'\colon T\to T$ such that $P\subset P'$, $f|_P=f'|_P$ and
$\rho(f,f')<\eps$.
\end{lem}
\begin{proof}

The proof is similar to the proof of Lemma~\ref{lem:small perturb}.
First, periodic points are dense in $T$, so again we may assume that $P$-basic intervals have diameters smaller than
some fixed $\delta>0$.
We denote by $I_j=[p_j,z_j]$ (resp. $I_j'=[p_j,z_j']$) the $P$ basic
interval in $C_j$ (resp. $T'$) containing $p_j$ (there is only one P-basic interval with containing $p_j$ in $T'$ and only one in $C_j$ since $p_j$ is an end point of both $T'$ and $C_j$ respectively).

Fix any points $r,s\in (p_0,z_0)$ and $r',s'\in (p_0,z_0')$ and assume that $p_0<r<s<z_0$, $z_0'<s'<r'<p_0$  with respect to ordering in $[p_0,z_0]$ and $[p_0,z_0']$, respectively.
We define $f'(x)=f(x)$ everywhere except at points of the open arcs $(p_0,z_0)$ and $(p_0,z_0')$. We define $f'(r)=z_1'$, $f'(s)=p_1$ and make $f'$ linear on
intervals $[p_0,r]$, $[r,s]$, $[s,z_0]$. Similarly, we put $f'(r')=z_1$, $f(s')=p_1$ and make $f'$ linear on
intervals $[p_0,r']$, $[r',s']$, $[s',z_0']$.

Denote $P'=P\cup \set{s,s',r,r'}$ and observe
that if we consider a subgraph in the $P'$-Markov graph of $f'$ restricted to $P'$-basic intervals in $T'$ then it is strongly connected,
since $P$-Markov graph of $f$ restricted to $P$-basic intervals in $T'$ was strongly connected.
Similarly, graph defined by $P'$-basic intervals in $\bigcup_{i=0}^{n-1} C_i$ is strongly connected.
But there are also edges joining these two subgraphs together in $P'$-Markov graph of $f'$ so this graph is strongly connected too.
This shows that $f'$ is transitive by Lemma~\ref{lem:Markov-transitive}. If it is not totally transitive, we can make it mixing by application of Lemma~\ref{lem:small perturb}.
But since we can control the size of $P$-basic intervals when starting our construction, and since we can also control the range of perturbation
in Lemma~\ref{lem:small perturb}, the condition $\rho(f,f')<\eps$ can be easily satisfied.
\end{proof}

%

\begin{cor}\label{cor:dyn_on_star}
 For any $n>0$, any $\eps>0$ and any $n$-star $T$ there is a mixing map $f\colon T\to T$ such that endpoints as well as the branch points are fixed points for $f$
and furthermore $\rho(f,\id)<\eps$.
\end{cor}
\begin{proof}
The result is a direct consequence of Lemma~\ref{lem:tree_attach}. First fix any finite set $P$ and split $T$ into sufficiently small $P$-basic intervals.
On each such interval define a transitive piecewise linear interval Markov map which fixes endpoints (e.g. the three fold map $f(0)=f(2/3)=0$, $f(1/3)=f(1)=1$ ).
Start with a basic interval containing an endpoint of star and then apply Lemma~\ref{lem:tree_attach} recursively, attaching adjacent intervals to it one-by-one
(creating a bigger tree in each step). If the initial intervals were small, and also the perturbation was
sufficiently small in each application of Lemma~\ref{lem:tree_attach} then at the end the map is an $\eps$-perturbation of identity.
By the construction it is totally transitive and also each endpoint is a fixed point, hence the map must be mixing.
\end{proof}

\begin{thm}\label{mix:constr}
Let $D$ be a dendrite with dense set of endpoints and such that degree of any branch point is equal to $n\geq 3$. Then
there exists a mixing map on $D$ such that every branch point in a periodic orbit and every arc in
$D$ contains a point which has a dense orbit and is not an endpoint.
\end{thm}
\begin{proof}
Fix an $\eps>0$ and put $\eps_i=\eps/2^i$.
Let $\mathcal{R}=\set{r_1,r_2,\ldots}$ be a sequence consisting of all branch points in $D$.
Take any two distinct endpoints of $D$ and denote by $T_1$ the arc between them. We will perform a recursive construction.
Let $P_1\subset T_1$ be a finite set containing $\End(T_1)$
and let $f_1$ be a $P_1$-Markov, $P_1$-linear and mixing map on $T_1$, for example $f_1$ can be the standard tent map.
Note that for any branch point $r$ of $D$ contained in $T_1$ the number of connected components of $D\setminus \set{r}$
not intersecting $T_1$ is equal to $n-2$.
There exists $m_1$ such that for any arc $I$ in $T_1$ with diameter at least $\eps_1$ we have that
$d_H(f_1^{m_1}(I),T_1)<\frac{\eps_1}{2}$, i.e. each such $I$ after $m_1$ iterations covers $T_1$ except maybe
arcs of diameter $\eps_1/2$ at endpoints of $T_1$ (here $d_H$ denotes Hausdorff distance induced by $d$). There also exists $\delta_1>0$ such that if $h\colon D\to D$ is a map satisfying $\rho(h|_{T_1},f_1)<\delta_1$
then $d_H(h^{m_1}(I),T_1)<\eps_1$, provided that $I$ is an arc in $T_1$ with diameter at least $\eps_1$.
There exists $\gamma_1<\delta_1/8$ such that if $g_1\colon T_1\to T_1$ is a homeomorphism satisfying $\rho(g_1,\id_{T_1})<\gamma_1$ then
$\rho(g_1\circ f_1 \circ g_1^{-1},f_1)<\delta_1/4$ and $g_1\circ f_1 \circ g_1^{-1}$ is $g_1(P_1)$-Markov. If $g_1$ is $P_1$-linear then $g_1\circ f_1 \circ g_1^{-1}$ is $g_1(P_1)$-linear.

Take the first branch point $r$ in the sequence $\mathcal{R}$ which belongs to $T_1$. Since $f_1$ is mixing, there exists a periodic point $p$
such that $d(p,r)<\gamma_1$ and $p$ is not a branch point in $T_1$ (in fact, it can be arbitrary at this step, since $T_1$ is an arc).
We assume that $\dist(\orbp(p),r)=d(p,r)$.
There exist branch points $q_1,\ldots,q_k$ of $D$ contained in $T_1\setminus \End(T_1)$, where $k=\# \orbp(p)$, such that $d(q_i,f^i(p))<\gamma_1$
and there is no other point $q_j$ in the arc $[q_i, f^i(p)]$ for $1 \leq j\neq i\leq k$, and none of point $q_i$ is a branch point of $T_1$. If we add all points $f^i(p)$ to the original partition of $T_1$ then $f_1$ remains Markov and linear with respect to that new extended partition. But now, if we use piecewise linear map $g_1\colon T_1\to T_1$ that sends each vertex $f^i(p)$ to $q_i$ and fixes points in $P_1$, then $g_1$ is $\gamma_1$ perturbation of $\id_{T_1}$ and by the definition is $P'$-linear, where $P'=P_1\cup \set{f^i(p) : i\geq 0}$. If we put $\tilde{f_1}=g_1\circ f_1 \circ g_1^{-1}$, then $\rho(\tilde{f_1},f_1)<\delta_1/4$ and $\tilde{f_1}$ is $g_1(P')$-Markov and $g_1(P')$-linear, but now points
$q_1,\ldots,q_k$ form a periodic orbit. Note that the image of any other point in the original partition does not change.
In particular, if there was a periodic orbit in $P_1$, it remains a periodic orbit. Note that $\tilde{f_1}$ is mixing, since it is conjugate with a mixing map.

In each connected component of $T_1\setminus \set{q_i}$ we have $n-2$ connected components not intersecting $T_1$.
Taking in each such component an arc from $q_i$ to an endpoint of $D$ we obtain arcs $L_i^{0},\ldots,L_i^{n-2}$.

We are going to define a finite set $Q_1$ and construct a transitive $Q_1$-Markov, $Q_1$-linear map $h$, sending arcs $L_i^j$ into a cycle in such a way that $h$ is a permutation on $\set{q_i : 1\leq i \leq k}$ and $\rho(h,\pi)<\frac{\delta_1}{8}$, where $\pi\colon \bigcup L_i^j \to \bigcup L_i^j$ is a map such that
$\pi|_{L_i^j}\colon L_i^j \to L_{\hat i}^{\hat{j}}$ is linear homeomorphism, where $\hat{i}=i+1$, $\hat{j}=j$ when $i<k$ and $\hat{i}=1$, $\hat{j}=j+1 \;(\text{mod }n-1)$ otherwise.

Denote $\Lambda=[0,1]\times \set{1,\ldots,k}\times\set{0,\ldots,n-2}$
Identify each $L_i^j$ (by a linear homeomorphism) with the interval $[0,1]\times \set{i,j}$, where $q_i=(0,i,j)$ for each $j$, and fix any $\xi>0$.
By Corollary~\ref{cor:dyn_on_star} there exists a mixing map $\phi\colon [0,1]\to [0,1]$
such that $\rho(\phi,\id)<\xi$. Define a map $h\colon \Lambda\to \Lambda$ by
$$
h(x,i,j)=\begin{cases}
(x,i+1,j) &, i<k\\
(\phi(x),1,j+1 (\text{mod }n-1)) &, i=k\\
\end{cases}
$$
The map $h \colon \Lambda \to \Lambda$ is well defined transitive map.
Now it is enough to identify each $L_i^j$ with $([0,1]\times \set{i,j})$ by linear homeomorphism
sending $q_i$ to $[(0,i,j)]$ and denote by $h$ the map obtained after this identification.
Since $\xi$ can be arbitrarily small, condition $\rho(h,\pi)<\frac{\delta_1}{8}$ is satisfied.

Put $T_2=T_1\cup \bigcup L_i^j$ and let $f_2$ be a mixing map on $T_2$ obtained by Lemma~\ref{lem:tree_attach} such that $\rho(f_2,\tilde{f}_1\cup h)<\delta_1/8$. Note that each $q_i$ is a branching point in $T_2$.
By the definition $f_2$ is $P_2$-Markov and $P_2$-linear with respect so some finite set $P_2\supset P_1$. By the construction
$\rho(f_1,(f_2)_{\mid T_1})<\delta_1/4$.

Applying the above procedure, we can inductively construct a sequence of trees $T_1\subset T_2 \subset \ldots$
such that $\overline{\bigcup T_i}=D$ (because each branch is eventually included in $T_i$ for $i$ large enough), together with
$\delta_i>0$, $m_i>0$ and mixing maps $f_i\colon T_i\to T_i$ such that:
\begin{enumerate}
\item\label{thm10:c1} For any $j>i$ we have $\rho(f_j,f_i)< \delta_i \sum_{s=i}^{j-1} 2^{-(s+1)}<\delta_i/2$.
\item\label{thm10:c2} For any arc $I$ in $T_j$ with diameter at least $\eps_j$ we have that
$d_H(h^{m_j}(I),T_j)<\eps_j$ provided that $\rho(h|_{T_j},f_j)<\delta_j$.
In particular $d_H(f_i^{m_j}(I),T_j)<\eps_j$ provided that $i\geq j$.
\item\label{thm10:c3} For each branch point $r_n$ of $D$ there is $N>0$ such that $r_n$ is a periodic point of $f_N$ and $f_j^k(r_n)=f_N^k(r_n)$
for all $j\geq N$ and $k\geq 0$.
\end{enumerate}

Note that using first point maps $R_i\colon D \to T_i$ (retractions) we obtain a Cauchy sequence $f_i \circ R_i$,
hence there exists a map $F\colon D\to D$ such that $f_i\circ R_i$ converges to $F$ in the metric $\rho$ (i.e. we have uniform convergence).
By \eqref{thm10:c1} we obtain that $\rho(F|_{T_i}, f_i)\leq \delta_i/2 <\delta_i$, hence by \eqref{thm10:c2} for any arc $J$ in $D$ and any $j$ there exists an arc $I\subset J$ and $m_j$ such that $d_H(f_i^{m_j}(I),T_j)<\eps_j$
for any $i>j$. In particular $d_H(F^{m_j}(I),T_j)\leq \eps_j$. Constructing a nested sequence of arcs, it is not hard to
show that $J$ contains a point with dense orbit.

In order to show that $F$ is weakly mixing, fix any nonempty open sets $U_1,U_2,U_3,U_4$.
There is $i>0$ such that $T_i \cap U_s\neq \emptyset$ for $s=1,2,3,4$.
For $s=1,2,3,4$ fix any arc $I_s\subset T_i\cap U_s$ and let
$$
\delta=\min_{1\leq s\leq 4}\set{\diam(I_s)}.
$$
Take $j>i$ sufficiently large, so that $\delta>3\eps_j$.
By the definition each $I_s\subset T_j$ and so \eqref{thm10:c2} we obtain that $\dist(F^{m_j}(I_s),I_{s'})<\eps_j$ for all $s,s'=1,2,3,4$,
which in particular implies that $F^{m_j}(U_1)\cap U_2\neq \emptyset$ and $F^{m_j}(U_3)\cap U_4\neq \emptyset$.
Indeed $F$ is weakly mixing, hence $F$ is
mixing by Theorem~\ref{tt:mix}.

\end{proof}

\section{Periodic points on dendrites and other continua}

It is easy to verify that if $f$ has dense periodic points then so does the induced map $2^f$.
It is also not surprising that the converse implication does not hold in general.
In this section we will show that on some dendrites
such a converse implication is valid (in fact, as we will see, it is valid on a wider class of continua containing some dendrites).

For the reader convenience we present here a standard, still very useful fact (see \cite[Theorem 8.16]{nadler92} and \cite[Theorem 8.23]{nadler92}).
\begin{lem}\label{lem:nad}
Let $X$ be a continuum, let $\alpha\subset X$ be an arc and let $g\colon \alpha\to X$ be continuous.
Then $g(\alpha)$ is arcwise connected and locally connected.
\end{lem}

Now we are ready to prove the following.

\begin{thm}\label{densep1}
Let $X$ be a continuum and $\fun{f}{X}{X}$ be a continuous map such that
$\mbox{Per}(2^f)$ is dense in $2^X.$ Let $[a,b]$ be a free arc in $X$ with
$a \ne b.$ Then $\mbox{Per}(f) \cap (a,b) \ne \emptyset.$
\end{thm}
\begin{proof}
Fix $p \in (a,b).$ Let us consider that in the natural order $<$ of the arc $[a,b]$ from $a$ to $b,$
we have $a < p < b.$ Take $a_1,b_1 \in [a,b]$ so that $a < a_1 < p < b_1 < b.$ Since the open arcs $(a_1,p)$ and
$(p,b_1)$ are open sets in $X$ and $\mbox{Per}(2^f)$ is dense in $2^X,$ there exist $A_1 \in \langle (a_1,p)\rangle
\cap \mbox{Per}(2^f)$ and $A_2 \in \langle (p,b_1) \rangle \cap \mbox{Per}(2^f).$ Then $A_1,A_2 \in 2^X,$
$A_1 \subset (a_1,p) \subset [a_1,b_1] \subset (a,b),$ $A_2 \subset (p,b_1) \subset [a_1,b_1] \subset (a,b)$
and there exist $n_1,n_2 \in \m{N}$ such that
$$
(2^f)^{n_1}(A_1) = f^{n_1}(A_1) = A_1 \hspace{.3cm} \mbox{and} \hspace{.3cm}
(2^f)^{n_2}(A_2) = f^{n_2}(A_2) = A_2.
$$
Let $n$ be the lowest common multiple of $n_1$ and $n_2.$ Then $f^n(A_1) = A_1$ and $f^n(A_2) = A_2.$
Note that if $A_1$ is finite then $(f^{n_1})_{|A_1} \colon A_1 \to A_1$ is a permutation of
the elements of $A_1$ and since every permutations of finite sets have
finite order, there exists $k \in \m{N}$ such that $(f^{n_1}|_{A_1})^k = \id_{A_1}.$ Hence
$A_1 \subset \Per(f)$ and then $\Per(f) \cap (a,b) \ne \emptyset.$ Similarly, if $A_2$ is
finite then, proceeding as before it follows that $\Per(f) \cap (a,b) \ne \emptyset.$ Hence we may
assume that both sets $A_1$ and $A_2$ are infinite.
Define
$$
c_1 = \min (A_1), \hspace{.2cm} d_1 = \max (A_1), \hspace{.2cm}  c_2 = \min (A_2) \hspace{.2cm}
\mbox{and} \hspace{.2cm} d_2 = \max (A_2).
$$
Then $c_1,d_1 \in A_1 \subset [c_1,d_1] \subset (a_1,p),$ $c_2,d_2 \in A_2 \subset [c_2,d_2] \subset (p,b_1)$
and $c_1 < d_1 < p < c_2 < d_2.$

We claim that
\begin{enumerate}
\item\label{thmg1:1} there exist $c_1^\prime,d_1^\prime \in [c_1,d_1]$ such that $(f^n)(c_1^\prime) = c_1,$
   $(f^n)(d_1^\prime) = d_1$ and $(c_1^\prime,d_1^\prime) \cap (f^n)^{-1}(\{c_1,d_1\}) = \emptyset.$
\end{enumerate}
To show \eqref{thmg1:1} let
$$
x = \max \left((f^n)^{-1}(\set{c_1}) \cap [c_1,d_1]\right) \hspace{.3cm} \mbox{and} \hspace{.3cm}
y = \min \left((f^n)^{-1}(\set{d_1}) \cap [c_1,d_1]\right).
$$
Then $x \ne y.$ If $x < y,$ we define $c_1^\prime = x$ and $d_1^\prime = y.$ Then $c_1^\prime < d_1^\prime$
and, by the definitions of $x$ and $y,$ the points $c_1^\prime$ and $d_1^\prime$ satisfy \eqref{thmg1:1}. Now assume that
$y < x.$ Define
$$
c_1^\prime = \min \left((f^n)^{-1}(\set{c_1}) \cap [y,x]\right) \hspace{.3cm} \mbox{and} \hspace{.3cm}
d_1^\prime = \max \left((f^n)^{-1}(\set{d_1}) \cap [y,c_1^\prime]\right).
$$
Then $d_1^\prime < c_1^\prime$ and the points $c_1^\prime,d_1^\prime$ satisfy \eqref{thmg1:1}. This completes the proof of \eqref{thmg1:1}.

Proceeding as in the proof of \eqref{thmg1:1}  we can show that:
\begin{enumerate}
\stepcounter{enumi}
\item\label{thmg1:2} there exist $c_2^\prime,d_2^\prime \in [c_2,d_2]$ such that $(f^n)(c_2^\prime) = c_2,$
   $(f^n)(d_2^\prime) = d_2$ and $(c_2^\prime,d_2^\prime) \cap (f^n)^{-1}(\{c_2,d_2\}) = \emptyset.$
\end{enumerate}
By the properties of the points $c_1^\prime$ and $d_1^\prime$ given in \eqref{thmg1:1}, Lemma~\ref{lem:nad} and the fact that $[c_1,d_1]$ is a
free arc in $X,$ the set $f^n\left([c_1^\prime,d_1^\prime]\right)$ satisfies one of the following two
conditions:
\begin{enumerate}
\setcounter{enumi}{2}
\item\label{thmg1:1.1} $f^n\left([c_1^\prime,d_1^\prime]\right) = [c_1,d_1];$
\item\label{thmg1:1.2} $f^n\left([c_1^\prime,d_1^\prime]\right)$ is a locally connected subcontinuum of $X$
     such that $c_1,d_1 \in f^n\left([c_1^\prime,d_1^\prime]\right)$
     and $f^n\left([c_1^\prime,d_1^\prime]\right) \cap (c_1,d_1) = \emptyset.$
\end{enumerate}
Similarly the set $f^n\left([c_2^\prime,d_2^\prime]\right) = [c_2,d_2]$ satisfies one of the following two
conditions:
\begin{enumerate}
\setcounter{enumi}{4}
\item\label{thmg1:2.1} $f^n\left([c_2^\prime,d_2^\prime]\right) = [c_2,d_2];$
\item\label{thmg1:2.2} $f^n\left([c_2^\prime,d_2^\prime]\right)$ is a locally connected subcontinuum of $X$ such that
     $c_2,d_2 \in f^n\left([c_2^\prime,d_2^\prime]\right)$ and $f^n\left([c_2^\prime,d_2^\prime]\right) \cap
     (c_2,d_2) = \emptyset.$
\end{enumerate}

Assume that $f^n\left([c_1^\prime,d_1^\prime]\right) = [c_1,d_1].$ Since $[c_1^\prime,d_1^\prime]$
is an arc contained in the arc $[c_1,d_1]$ such that $(f^n)(c_1^\prime) = c_1$ and $(f^n)(d_1^\prime) = d_1,$
there exists a fixed point of $f^n$ in $[c_1^\prime,d_1^\prime].$ This implies that $\mbox{Per}(f) \cap (a,b) \ne \emptyset.$
If $f^n\left([c_2^\prime,d_2^\prime]\right) = [c_2,d_2]$ then, proceeding as before, there exists
a fixed point of $f^n$ in $[c_2^\prime,d_2^\prime]$ and then $\mbox{Per}(f) \cap (a,b) \ne \emptyset.$

Let us assume now that conditions \eqref{thmg1:1.2} and \eqref{thmg1:2.2} hold. By \eqref{thmg1:1.2} and Lemma~\ref{lem:nad}
the set $f^n\left([c_1^\prime,d_1^\prime]\right)$ contains an arc $B_1$ with endpoints $c_1$ and $d_1$ so that
$B_1 \subset X \setminus (c_1,d_1).$ Since $[a,b]$ is a free arc in $X$ and the arcs in $[a,b]$ are unique, the subarc $[d_1,b]$ is
contained in $B_1.$ In particular
$$
[c_2^\prime,d_2^\prime] \subset [d_1,b] \subset B_1 \subset f^n\left([c_1^\prime,d_1^\prime]\right).
$$
Proceeding as before, using now \eqref{thmg1:2.2} we infer that there exist an arc $B_2$ with endpoints $c_2$ and $d_2$ so that
$B_2 \subset X \setminus (c_2,d_2).$ Hence
$$
[c_1^\prime,d_1^\prime] \subset [a,c_2] \subset B_2 \subset f^n\left([c_2^\prime,d_2^\prime]\right).
$$
It is not hard to see that
there exist $c_1^\star,d_1^\star \in [c_1^\prime,d_1^\prime]$ such that $f^n\left([c_1^\star,d_1^\star]\right) =
[c_2^\prime,d_2^\prime].$ There also exist $c_2^\star,d_2^\star \in [c_2^\prime,d_2^\prime]$ so that
$f^n\left([c_2^\star,d_2^\star]\right) = [c_1^\star,d_1^\star].$ Hence
$$
[c_1^\star,d_1^\star] = f^n\left([c_2^\star,d_2^\star]\right) \subset f^n\left([c_2^\prime,d_2^\prime]\right)
= f^n\left(f^n\left([c_1^\star,d_1^\star]\right)\right) = f^{2n}\left([c_1^\star,d_1^\star]\right)\subset (a,b).
$$
From this it follows that there is a fixed point of $f^{2n}$ in $[c_1^\star,d_1^\star].$ Hence
$\mbox{Per}(f) \cap (a,b) \ne \emptyset.$
\end{proof}


Given a continuum $X$ let $\mathcal{FA}(X) = \bigcup \left\{\inte{J} \colon J \text{ is a free arc in } X \right\}$.
In [8], $X$ is defined to be \textit{almost meshed} if the set $\mathcal{FA}(X)$ is dense in $X.$ The class of almost meshed continua contains several interesting examples, among them dendrites with their
set of endpoints closed also dendrites with countable set of endpoints. Using this terminology, the following result can be proved by using Theorem 4.2.

\begin{thm}\label{densep2}
Let $X$ be an almost meshed continuum and $\fun{f}{X}{X}$ be a continuous map. If
$\mbox{Per}(2^f)$ is dense in $2^X,$ then $\mbox{Per}(f)$ is dense in $X.$
\end{thm}
\begin{proof}
Let $U$ be a nonempty open subset of $X.$ Since $X$ is almost meshed, the set
$\mathcal{FA}(X)$ is dense in $X.$ Then there exists $p \in U \cap \mathcal{FA}(X).$ Let
$J$ be a free arc in $X$ such that $p \in \inte{J}.$ Then $p \in U \cap \inte{J},$ so
there exist $a,b \in U \cap \inte{J}$ such that $p \in (a,b) \subset [a,b] \subset
U \cap \inte{J}.$ Then $a \ne b$ and, by Theorem \ref{densep1}, $\mbox{Per}(f) \cap (a,b) \ne
\emptyset.$ This implies that $\mbox{Per}(f) \cap U \ne \emptyset,$ so $\mbox{Per}(f)$ is
dense in $X.$
\end{proof}

Since finite graphs are almost meshed continua, as a corollary of Theorem \ref{densep2}
we obtain the following result.

\begin{cor}\label{densep3}
Let $X$ be a finite graph and $\fun{f}{X}{X}$ be a continuous function. If
$\mbox{Per}(2^f)$ is dense in $2^X,$ then $\mbox{Per}(f)$ is dense in $X.$
\end{cor}

Corollary \ref{densep3} appeared first in \cite[Theorem 3.3, p. 474]{gengrong}. However the proof
seems to be incomplete. As it is presented in \cite{gengrong}, it works fine for \emph{trees}, i.e. finite
graphs that contains no copies of a circle.

\begin{cor}\label{densep3}
Let $(X,f)$ be a dynamical system on a dendrite $X$ with either countable or closed set of endpoints (i.e. $\End(X)=\cl{\End(X)}$). If
$\mbox{Per}(2^f)$ is dense in $2^X,$ then $\mbox{Per}(f)$ is dense in $X.$
\end{cor}

\section*{Acknowledgements}
The authors express many thanks to Dominik Kwietniak, Logan Hoehn, Habib Marzougui and Chris Mouron for fruitful discussions
on periodicity and mixing in maps on dendrites.

The research of G. Acosta and R. Hern\'andez-Guti\'erez was partially supported by the project "Hiperespacios topol\'ogicos (0128584)" of CONACYT, 2009, and the project "Teor\'ia de Continuos, Hiperespacios y Sistemas Din\'amicos" (IN104613) of PAPIIT, DGAPA, UNAM.
The research of I. Naghmouchi was supported by the research laboratory: syst\`emes dynamiques et combinatoire: 99UR15-15.
The research of P. Oprocha was supported by the Polish Ministry of Science and Higher Education from sources
for science in the years 2013-2014, grant  no. IP2012~004272.

\end{document}